\documentclass{article}
\usepackage[T2A]{fontenc}
\usepackage[cp1251]{inputenc} %for Windows
\usepackage[english]{babel}
\usepackage{amsfonts,amssymb,mathrsfs,amscd}
\usepackage{amsthm,graphicx}

\newtheorem*{theorem}{Theorem}

\theoremstyle{definition}

\def\R{{\mathbb R}}

\title{Octagon and tropical octagon yield braid invariants}
\author{Vassily Olegovich Manturov}

\begin{document}

\maketitle

\begin{abstract}
In the present paper, we construct an invariant
of braids in the real projective plane which 
corresponds to an ``action'' of braids on certain
graphs in $\R{}P^{2}$ with labels. 
This paper is a sequel of papers \cite{M},\cite{KM}.

It demonstrates once more that solutions to the octagon relation in various forms
give rise to invariants of braids.

\end{abstract}



\section{Introduction}
Braids correspond to dynamical systems of points moving on a
$2$-surface (usually, on the plane, but here we
consider the projective plane $\R{}P^{2}$ following \cite{M}).

It often happens that the same equation appears in
various areas of mathematics. Once this is established,
results from one area often give rise to nice constructions in the other.

In \cite{InvariantsAndPictures} (see page 310) the author demonstrated
how solutions to the pentagon equation can be used for
constructing invariants of braids. One more invariant
was constructed by I.Rohozhkin \cite{Roh}.

In \cite{M}, we did a similar thing by using the 
{\em octagon equation}, making braids act on
configurations of lines and points (by using the Desargues
theorem \cite{FP}). We shall not write the octagon equation
but rather draw it.

In the present paper, we construct a simpler
invariant of braids from the octagon relation: the object
which is operated on by braids is simpler than configurations of lines and points (as done in \cite{M}).

The proof of the main theorem here almost repeats that from \cite{M}, however, the main construction does not require
``black'' and ``white'' dots corresponding to 
points and lines, which allows one to consider
not only braids on evenly many strands
but braids with arbitrary number of strands.

It seems that passing from braids in the projective
plane to braids in $\R^{2}$ is a technical issue,
which we shall consider elsewhere.

Let us consider braids on $\R{}P^{2}$ and, by using duality, we shall move
$n$ pairwise distinct lines. Generically, these lines intersect in $n\choose 2$ points.
By Euler characteristic reasons, these lines split the projective plane into
${n \choose 2}+1$ regions.

As points move, whenever no three of them
are collinear, their dual lines  
form a four-valent graph in $\R{}P^{2}$ (in general
position). The dual graph is naturally a 
quadrangulation of $\R{}P^{2}$ (splitting
of $\R{}P^{2}$ into quadrilaterals, see Fig. \ref{below}.

Fix a field $F$.
With vertices of this  graph we
associate formal variables $a,b,\cdots$.
We can consider the formal tropical
field $F(a,b,c,\dots)$ over $F$ generated by
$a,b,c,\cdots$
Later on, a braid will produce other
graphs from the initial one, and we shall 
associate elements from $F(a,b,c,\dots)$ to it.

\begin{figure}
\centering\includegraphics[width=150pt]{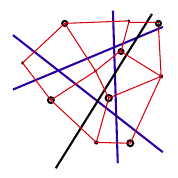}
\caption{Line configuration and quadrangular tiling}
\label{below}
\end{figure}

As lines move generically, the dual graph undergoes a hexagonal flip,
see \footnote{Some pictures are kindly borrowed from the paper
\cite{FP}.} Fig. \ref{Desarguesflip}).

\begin{figure}
\centering\includegraphics[width=150pt]{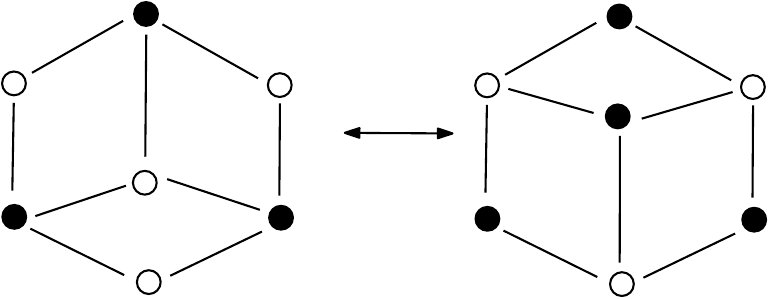}
\caption{The Desargues flip}
\label{Desarguesflip}
\end{figure}

Instead of the existing vertex labeled by $x$
we obtain a new vertex which we denote by
$x$ so that the label $x'$ of the new vertex
is expressed as follows:

$$x\to x'= \frac{ad+be+cf}{x}.
\eqno(1)$$

We call this transformation the {\em Desargues transformation} following \cite{FP}.

This equation is a partial case of the
following equation in a tropical field
$(F,\otimes,\oslash, \oplus)$, where
$a\otimes b = a+b, a\oslash b=a-b, a\oplus b=max(a,b)$.

In this case

$$x'=max(a+d-x,b+e-x,c+f-x) \eqno(2).$$

The general case which contains both (1) and (2) is
given by the formula

$$x\to x'=(a\otimes d\oplus b\otimes e \oplus c\otimes f) \oslash x).
\eqno(3)$$
\section{Acknowledgements}
I am very grateful to Igor Mikhailovich Nikonov, Louis Hirsch Kauffman and Seongjeong Kim for permanent
discussion of my current work.
I am extremely grateful to 
Ilya Rohozhkin for preparing this text.

\section{The main result}

Fix a configuration of points $z=\{z^{1},\cdots,z^{n}\}$
in $\R{}P^{2}$, which is generic
generic (with respect to \cite{M}).
A braid $\beta$ is considered as
a one-parametric family $z_{t}=\{z^{1}_{t},\cdots,z^{n}_{t}\},$
where $z^{j}_{t}$ move continuously as $t$ runs from $0$ to $1$, and
$z_{0}$ and $z_{1}$ coincide as sets.
Denote the projectively dual lines by
$l_{t}^{1},\cdots, l_{t}^{n}$.

We place generators  $a,b,\cdots$ of the tropical field in the points, vertices of the
dual graph. Denote this graph by $D_{n}$.

The typical non-generic situation is when three
points are collinear. In this case, the dual
lines pass through the same point.

When passing through this event, we invert the
triangle. From the point of view of the dual
graph, we get an inversion of the triangle which looks
as shown in Fig. \ref{invtriangle}.

We place variables $a,b,\cdots$ mentioned
above in the  vertices of the
dual graph.

Hence, we get the Desargues flip shown in Fig. 
\ref{Desarguesflip}.

\begin{figure}
\centering\includegraphics[width=250pt]{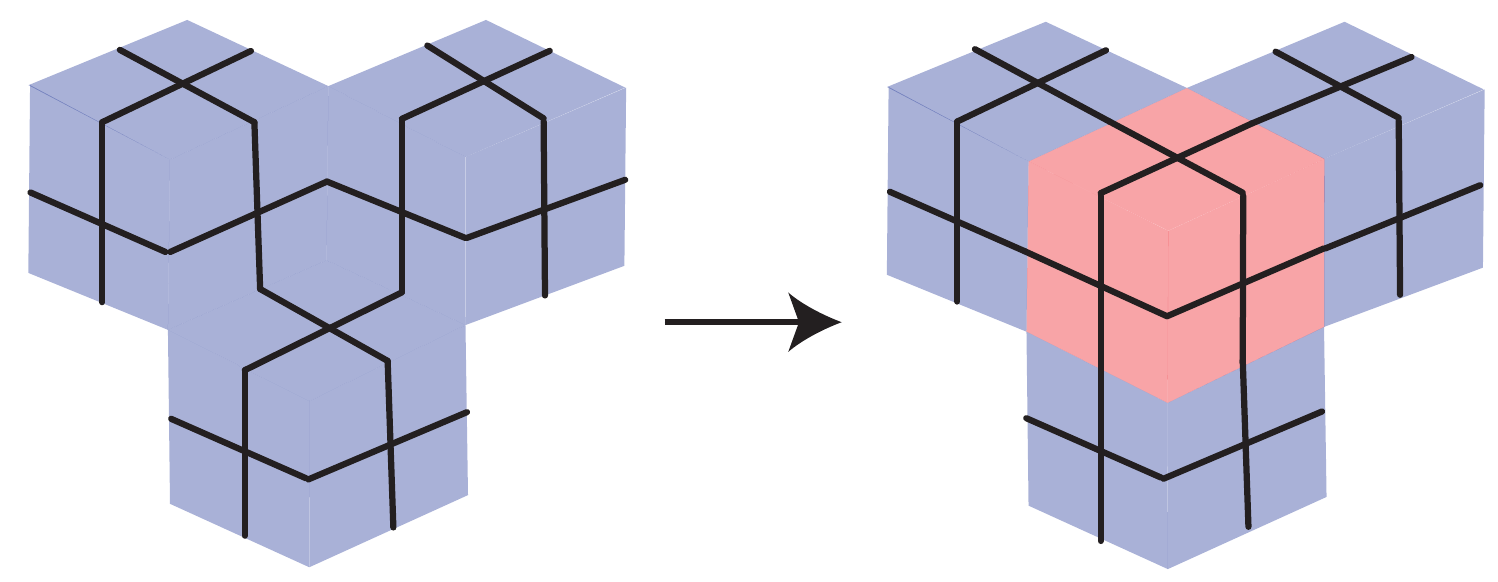}
\caption{Inverting the triangle}
\label{invtriangle}
\end{figure}

For each such flip we get a new vertex
$x'$ instead of $x$ which has to be labelled 
according to (2).

Hence, after applying all steps, we get
to the initial set of points $z^{1}=z^{0}$,
hence to the initial graph and the same dual
graph $D_{n}$ having some other labels.

\begin{theorem}
Isotopic braids give rise to identical transformations
of labels on the dual graph $D_{n}$.
\end{theorem}

\begin{proof}
The proof goes along the lines of \cite{M}.

One can consider a generic isotopy of braids. Standardly,
relations correspond to codimension 2 events (generators correspond to
codimension one events, which are the flips given above). This gives rise to three types
of transformations.

Consider two isotopic generic braids $\beta_{1}$ and $\beta_{2}$ for the same set of initial points $\{z\}$. 
We may assume that the isotopy between these
braids is {\em generic} in the sense of \cite{M}. 
The genericity of the isotopy means the following.
Standardly, when studying paths in the configuration
space, generators correspond to codimension 1 events.
In our case they are flips  and correspond
to triples of collinear points.
Similarly,
relations correspond to codimension 2 events.
Following \cite{M}, it suffices to consider
only three types of transformations of the isotopy.

One of them corresponds to the situation when we perform 
two Desargues flips in one direction and in the
opposite direction (see Fig. \ref{Desarguesflip2}) and return to the
initial labels because the maps
$x\to x'$ and $x'\to x$ are inverse to each other.

The other one (for more details see
\cite{M}) corresponds to the situation which can be characterised as
``independent events commute'': two flips occur in two hexagons having no common interior points can be performed in 
any sequence and give rise to the same result.
This is rather evident because variables taking
part in the equation change independently.

\begin{figure}
\centering\includegraphics[width=150pt]{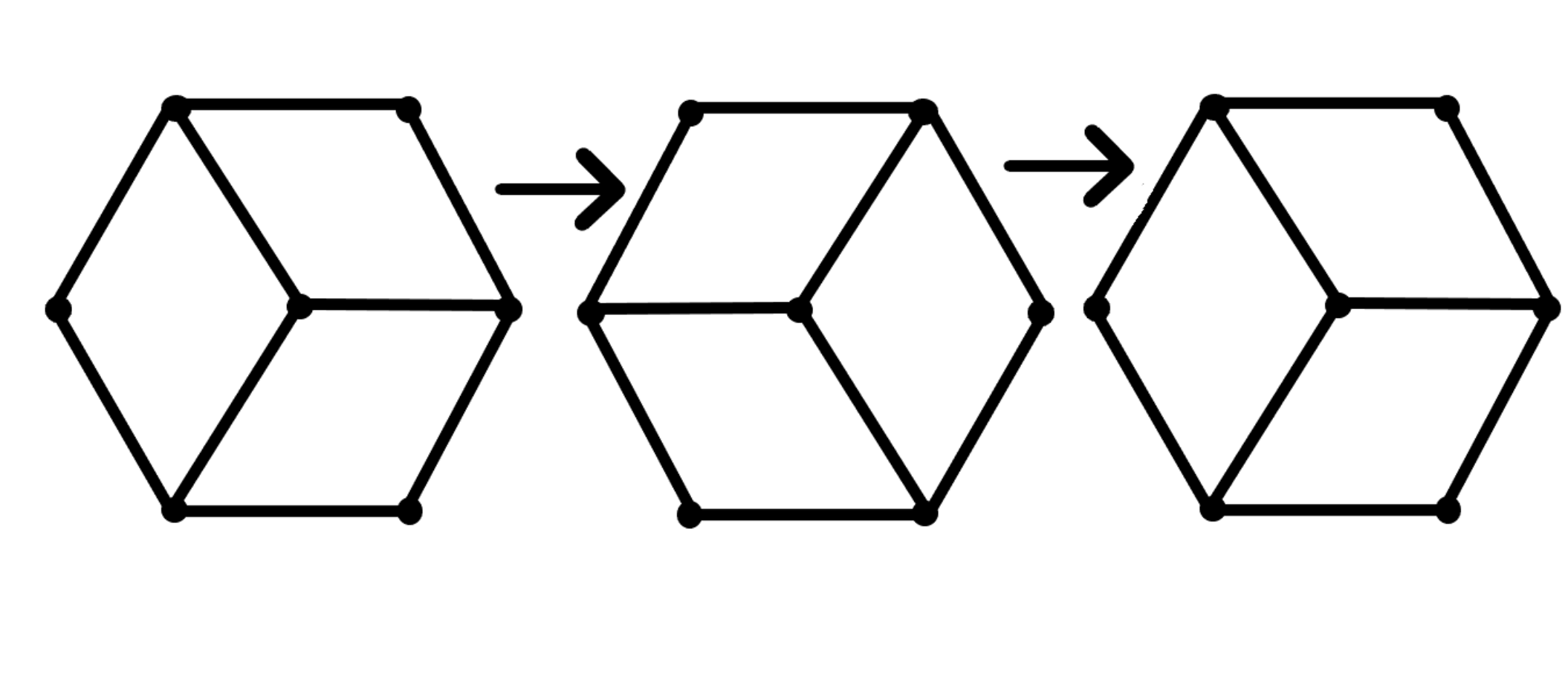}
\caption{The composition of two Desargues  flips}
\label{Desarguesflip2}
\end{figure}

The most important one is the octagon relation.
The main thing we have to check is that the transformation
(3) [and hence, (1) and (2)] satisfy the octagon relation.

This is actually shown by A.Enriques and D.Speyer \cite{ES};
the only thing we did in (\ref{enriquesspeyer})
is: we changed usual operations of multiplication,
divison, and addition by their tropical analogues.

\begin{figure}
\centering\includegraphics[width=350pt]{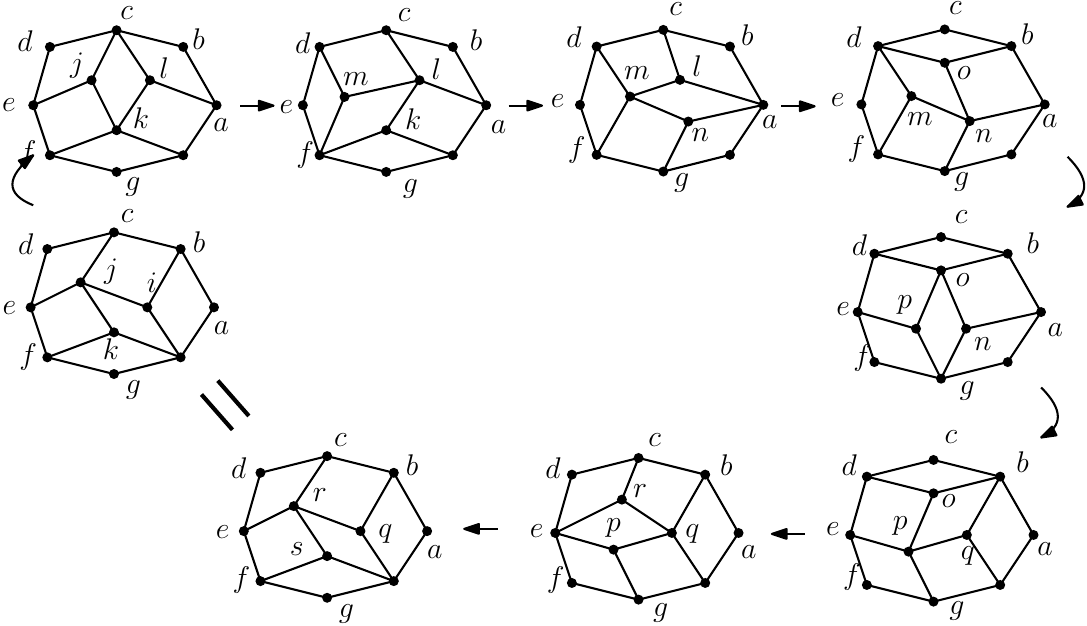}

\label{enriquesspeyer}
\end{figure}

Let us gradually perform the calculations for labels.

We have:
$$l=({a\otimes j\oplus b\otimes k\oplus c\otimes h})\oslash{i};$$

$$m=({c\otimes f\oplus d\otimes k\oplus e\otimes l})\oslash{j}$$

$$=({c\otimes f\otimes i\oplus d\otimes k\otimes i\oplus e\otimes a\otimes j\oplus e\otimes b\otimes k\oplus e\otimes c\otimes h})\oslash({i\otimes j});$$

$$n=({f\otimes a\oplus g\otimes l\oplus h\otimes m})\oslash{k}$$ 

$$=(f\otimes a\otimes i\otimes j\oplus g\otimes a\otimes j^2\oplus g\otimes b\otimes k\otimes j\oplus g\otimes c\otimes h\otimes j\oplus h\otimes c\otimes f\otimes i$$

$$\oplus
h\otimes d\otimes k\otimes i\oplus h\otimes e\otimes a\otimes j\oplus h\otimes e\otimes b\otimes k\oplus e\otimes c\otimes h^2)\oslash({i\otimes j\otimes k});$$

$$o=({a\otimes d\oplus b\otimes m\oplus c\otimes n})
\oslash{l}$$ $$=({b\otimes e\otimes k\oplus d\otimes 
i\otimes k\oplus c\otimes g\otimes j\oplus c\otimes 
f\otimes i\oplus c\otimes e\otimes h})\oslash({j\otimes k});$$

$$q=({h\otimes o\oplus a\otimes p\oplus b\otimes g})
\oslash{n}=i;$$

$$r=({b\otimes e\oplus c\otimes p\oplus d\otimes q})=j;$$

$$s=({e\otimes h\oplus f\otimes q\oplus g\otimes r})\oslash{p}=k.$$

This completes the proof.
\end{proof}

Certainly, once the construction work for {\em any
tropical field} (2), it also works for the classical
field of rational functions according to (3) and for the tropical field (4).

It is important to notice that there is a
{\em positivity phenomenon} in the cluster
algebra situation, saying that in the case of (3)
all labels will be not just rational functions but rather
{\em Laurent polynomials} in variables coresponding
to the initial graph.

\section{Further directions}

It is well known that invariants of braids can
be obtained from solutions of the Yang-Baxter
equations which correspond to the Reidemeister
moves.

In \cite{InvariantsAndPictures,Roh} the
author used the approach with Vorono\"{\i} 
diagrams
to get invariants of braids by using solutions
to the pentagon equations.

One of the main goals of the present paper as well
as \cite{M}
was: any solutions of octagon give braid invariants.

An important problem is to understand a deeper
relation between the Yang-Baxter, pentagon,
and octagon relations by using braid groups.

One more goal is to obtain invariants of 
knots and links from the invariants constructed
in \cite{M} and in the present paper.

\end{document}